\documentclass[12pt]{article}
\usepackage{amssymb,amsmath, amsfonts}
\usepackage{enumerate}
\usepackage{amsthm}
\usepackage{epsfig}
\usepackage{graphicx}
\usepackage{caption}
\usepackage{subcaption}
\usepackage[title]{appendix}
\usepackage{url}
\usepackage{float}
\usepackage{enumitem}

\usepackage{microtype}
\usepackage{autobreak}
\usepackage{stackengine}
\usepackage{environ}
\NewEnviron{myequation}{%
	\begin{align*}
	\scalebox{1.5}{$\BODY$}
	\end{align*}
}
\stackMath

\allowdisplaybreaks
\allowdisplaybreaks
\def\@biblabel#1{}

\makeatletter
\def\@biblabel#1{}
\makeatother

\newtheorem{theorem}{Theorem}[]

\newtheorem{remark}{Remark}[]
\newtheorem{lemma}{Lemma}[]

\begin{document}
	
	\begin{center}
		{\Large\bf  Estimation after selection from bivariate normal  population using LINEX	loss function } \\
	\vspace{0.3in}
{\bf Mohd. Arshad$^{a,}$\footnote{Corresponding author. E-mail addresses: ~arshad.iitk@gmail.com (M. Arshad), \\   ~abdalghani.amu@gmail.com
		(O. Abdalghani), kaluram.iitkgp@gmail.com
		(K.R. Meena).}  Omer Abdalghani$^{a}$ and \bf K.R. Meena $^{b}$}
\\
$^{a}$Department of Statistics and Operations Research, 
Aligarh Muslim University, Aligarh, India. \\
$^{b}$Department of Mathematics, Acharya Narendra Dev College,  Delhi, India.
	\end{center}
	
	\noindent------------------------------------------------------------------------------------------------------------
	\begin{abstract}
Let $\pi_1$ and $\pi_2$ be two independent populations, where the population $\pi_i$ follows a bivariate normal distribution with unknown mean vector $\boldsymbol{\theta}^{(i)}$ and  common known  variance-covariance matrix $\Sigma$, $i=1,2$. The present paper is focused on estimating a characteristic $\theta_{\textnormal{y}}^S$ of the selected bivariate normal population, using a  LINEX loss function. A natural selection rule is used for achieving the aim of selecting the best bivariate normal population. Some natural-type estimators and Bayes estimator (using a conjugate prior) of $\theta_{\textnormal{y}}^S$ are presented. An admissible subclass of equivariant estimators, using the LINEX loss function, is obtained. Further, a sufficient condition for improving the competing estimators of $\theta_{\textnormal{y}}^S$ is derived. Using this sufficient condition, several estimators improving upon the proposed natural estimators are obtained. Further, a real data example is provided for  illustration purpose. Finally, a comparative study on the competing estimators of $\theta_{\text{y}}^S$ is carried-out using simulation.

 	\end{abstract}
\vskip 2mm
	 \section{Introduction.}  
   The  estimation  of a characteristic after selection has been recognized as an important practical problem for many years. The problem  arises naturally in multiple applications where one wishes to select a population from the   available $k\, (\geq 2)$ populations and then estimate some characteristics (or parametric functions) associated with the population selected by a fixed selection rule. For example, in modelling economic phenomenons, often the economist is faced with the problem of choosing an economic model from $k\, (\geq 2)$  different models that returns a minimum loss to the capital economic. After the selection of the desired  economic model, using a pre-specified selection procedure,  the economist would like to have an estimate of the return losses from the selected model. In clinical research, after the selection of  the most effective treatment from a choice of $k$ available treatments, a doctor may wishes to have an estimate of the effectiveness of the selected treatment. The aforementioned problems are continuation of the general formulation of the Ranking and Selection problems. Several inferential methods for statistical selection and estimation related to these problems   have been developed by many authors, see
   Cohen and Sackrowitz (1982),    Misra and Dhariyal (1994), Misra and  van der Meulen (2001), Vellaisamy and Punnen (2002), Stallard et al. (2008), Vellaisamy (2009),   Misra and Arshad (2014), Arshad et al. (2015), Arshad and	Misra (2015a, 2015b), Fuentes et al. (2018),  Meena et al. (2018),  Arshad and Abdalghani (2019).

 The majority of prior studies on  selection and estimation following selection problems have exclusively focused on a selected univariate population, and very few papers have appeared for a selected  bivariate/multivariate population.
Some of the works devoted to the bivariate/multivariate case are due to
  Amini and Nematollahi (2016) and Mohammadi and Towhidi (2017).  
In particular,   Mohammadi and Towhidi (2017)  considered the estimation of a  characteristic  after selection from bivariate normal population, using a  squared error loss function. The authors used this loss function and derived a Bayes estimator  of a characteristic of the bivariate normal population selected by a natural selection rule. The authors also provided some admissibility and inadmissibility results. This paper continues the study of  Mohammadi and Towhidi (2017) by considering the following loss function 
\begin{equation} \label{loss1.1}
L (\delta , \theta) = e^{a \left(\delta - \theta \right) } - a (\delta - \theta) - 1 . ~ \ \ \delta \in \mathbb{D}, \ \theta \in \Theta,
\end{equation}  
where $\delta$ is an estimator of the unknown parameter $\theta$,  $a$ is a location parameter of the  loss function (\ref{loss1.1}),  $\Theta$ denotes the parametric space, and  $\mathbb{D}$ represents a class of estimators of $\theta$. The  loss function in Equation (\ref{loss1.1}) is generally called an asymmetric  linear exponential (LINEX) loss and is useful in situations where positive bias (overestimation) is assumed to  be more preferable than negative bias (underestimation) or vice versa. Many researchers have used the above loss function, see among others Zellner (1986), Lu et al. (2013)  Nematollahi and Jozani (2016), and  Arshad and Abdalghani (in press).

The normal distribution is the most important and  used probability model in  many natural phenomena. For instance,  variables such as psychological, educational, blood pressure, and heights, etc.,  follow normal distribution.
  One  generalization of the univariate normal distribution is the bivariate normal distribution.  
Consider two independent populations $\pi_1 $ and $\pi_2$. 
Let  $\boldsymbol{Z}_i =(X_i, Y_i)^\intercal$ be a random vector associated with the bivariate normal population $\pi_i \equiv N (\boldsymbol{\theta}^{(i)}, \boldsymbol{\Sigma})$, where $\boldsymbol{\theta^{(i)}}= \left(\theta_x^{(i)}, \theta_y^{(i)} \right)^\intercal $ denotes the 2-dimensional  unknown mean vector $(i=1,2)$, and $\boldsymbol{\Sigma} = \begin{bmatrix}
\sigma_{xx} & \sigma_{xy}  \\
\sigma_{xy} & \sigma_{yy}
\end{bmatrix}$ denotes the  common known positive-definite  variance-covariance matrix.
Suppose that the  $Y$-variate  is  a characteristic  which is difficult (or expensive) to measure whose mean is of interest, and  the  $X$-variate is an auxiliary  characteristic  which is easy (or inexpensive) to measure. Then, based on an available information of the $X$-variate, we wish to make some inferences about the corresponding $Y$-variate. For instance, $X$ may be the grade of an applicant  on a particular test and $Y$ is  regarded as  a grade  on a future test. Then, based on the $X$-grade we want to see   the  behavior of the corresponding Y-grade. 
  Let $X_{(1)}$ and $X_{(2)}$ be the order statistics from $X_1$ and $X_2$.
Then, the $Y$-variates induced by the order statistic $X_{(i)}$ is called the concomitant of $X_{(i)}$ and is denoted by $Y_{[i]}$ ($i=1,2$). Assume  that the bivariate population associated with $\max\{\theta_{x}^{(1)}, \theta_{x}^{(2)}\}$ is referred as the better population. For selecting the better population, a  natural selection rule $\boldsymbol{\psi}=(\psi_{1}, \psi_{2})$ selects the population associated with  $X_{(2)}=\max(X_1,X_2)$, so that, the natural selection rule $\boldsymbol{\psi}=(\psi_{1}, \psi_{2}) $ can be expressed as 
\begin{equation} \label{sel-rule}
\psi_{1}(\boldsymbol{x})=\left \{ \begin{array}{ll} 1,
& \mbox{if} \  \ X_1 > X_2 \\
0, & \mbox{if} \ \ X_1 \leq X_2, \end{array} \right.
\end{equation}
and $\psi_{2}(\boldsymbol{x}) = 1-\psi_{1}(\boldsymbol{x})$.  After a bivariate normal  population is selected using the selection rule $\boldsymbol{\psi}$, given in (\ref{sel-rule}), 
we are interested in the estimation of  the second component of the mean vector associated with the selected population, which can be expressed as
\begin{align*} \label{parameter}
\theta_{\text{y}}^S (\boldsymbol{x})&=  \theta_{y}^{(1)} \psi_{1}(\boldsymbol{x}) + \theta_{y}^{(2)} \psi_{2}(\boldsymbol{x}) \vspace{2mm}
\\
 & = \left \{ \begin{array}{ll} \theta_{y}^{(1)},
& \mbox{if} \  \ X_1 > X_2 \vspace{2mm} \\
\theta_{y}^{(2)}, & \mbox{if} \ \ X_1 \leq X_2. \end{array} \right.
\end{align*}
 Note that $\theta_{\text{y}}^S$  depends on the variable $X_i, \, i=1,2$, so that is  a random parameter. 
Our goal is to  estimate $\theta_{\text{y}}^S$ using  the loss function given in  (\ref{loss1.1}).  
 
  Putter and Rubinstein (1968) have shown that  an unbiased estimator    of the mean after selection from univariate normal population does not exist. 
 Dahiya (1974) continued the study of Putter and Rubinstein (1968) by  proposing several different estimators of mean and  investigated their corresponding bias and mean squared error.  Later, Parsian and Farsipour (1999)   considered   two univariate normal populations having same known variance but  unknown means, using the loss function given in (\ref{loss1.1}). They suggested seven different estimators for the mean and investigated their respective biases and risk functions. Misra and van der Muelen (2003) continued the study of Parsian and Farsipour (1999) by deriving some admissibility and inadmissibility results for estimators of the mean of the univariate normal  population selected by a natural selection rule. As a consequence, they obtained some estimators better than those suggested by Parsian and Farsipour (1999). Recently, Mohammadi and Towhidi (2017) extended  the study of  Dahiya (1974) by considering a bivariate normal population. The authors  derived Bayes and minimax estimators  and an admissible subclass of natural  estimators  were also obtained. Further, they  provided some improved estimators of  the mean of the selected bivariate normal population. 
  This article continues the investigation
   of  Mohammadi and Towhidi (2017) by deriving various competing  estimators and decision theoretic results  under the LINEX loss function. 
  
Note that,  using the loss function given in (\ref{loss1.1})  for  estimating  $\theta_{\text{y}}^S$, the estimation problem under consideration  is location invariant with regard to a  group of permutation and a location  group of  transformations. 
Moreover, its appropriate to use permutation and location invariante estimators satisfying  $\delta \left( \boldsymbol{Z}_1, \boldsymbol{Z}_2 \right)= \delta \left( \boldsymbol{Z}_2, \boldsymbol{Z}_1 \right) $ and $\delta\left( \boldsymbol{Z}_1 + \boldsymbol{c}, \boldsymbol{Z}_1+ \boldsymbol{c}\right)  =\delta\left( \boldsymbol{Z}_1 ,\boldsymbol{Z}_1  \right)+ c_2, \  \forall \ \boldsymbol{c}= \left( c_1,c_2 \right)^{\intercal}  \in \mathbb{R}^2$, 
where $\mathbb{R}^2$  denotes the 2-dimensional Euclidean space. Therefore, any location equivariant estimator of $\theta_{\text{y}}^S$ will be of the form 
\begin{equation}\label{equi-c}
\delta_\varphi  \left(  \boldsymbol{Z}_1, \boldsymbol{Z}_2 \right)= Y_{[2]} + \varphi \left( X_{(1)} - X_{(2)},  Y_{[1]}- Y_{[2]} \right),  
\end{equation} 
where $ \varphi(\cdot) $ is a function of $X_{(1)} - X_{(2)}$ and $Y_{[1]}- Y_{[2]}$. Let $\mathcal{Q}_c$ represents the class of all equivariant estiamtors of the form (\ref{equi-c}). 
  For notational simplicity, the following notations will be adapted throughout the paper; $\boldsymbol{Z}=(\boldsymbol{Z}_1,\boldsymbol{Z}_2)$,  
  $ \theta_x = \max \left(  \theta_x^{(1)}, \theta_x^{(2)} \right) -  \min \left( \theta_x^{(1)}, \theta_x^{(2)} \right) $, $\theta_y= \max \left(  \theta_y^{(1)}, \theta_y^{(2)} \right) -  \min \left( \theta_y^{(1)}, \theta_y^{(2)} \right)$,  $\boldsymbol{\theta}^*=\left(\theta_{x},  \theta_{y} \right)^\intercal \in  \mathbb{R}_+^2$, where $\mathbb{R}_+^2$  denotes  the positive part of the two dimensional
  Euclidean space $\mathbb{R}^2$, and  $\phi(\cdot) $ and $ \Phi(\cdot)$ denote the usual pdf and cdf of $N(0,1)$. 

 We presented  some natural estimators and  Bayes estimator, under the loss function (\ref{loss1.1}),   of $\theta_{\text{y}}^S$ in Section 2.  In Section 3, an admissible subclass of natural type estimator is obtained. 
  Further,  a result of improved estimators  is derived in Section 4. 
 In Section 5, a data analysis using a real data set is provided to illustrate the computation of the various estimates of $\theta_{\text{y}}^S$. Finally, in Section 6, using the LINEX loss function, risk comparison of the estimators of $\theta_{\text{y}}^S$  is carried-out using  a simulation study. 

\section{Estimators of $\theta_{\text{y}}^S$ }
In this section, we present various estimators of $\theta_{\text{y}}^{S}$ of the selected population. First, based on the maximum likelihood estimator (MLE), an estimator of $\theta_{\text{y}}^{S}$ is given by 
\begin{align*}
\delta_{N,1} (\boldsymbol{Z}) 
= Y_{[2]}.
\end{align*}
Similarly, based on the minimum risk equivariant estimator (MREE), an estimator of $\theta_{\text{y}}^{S}$ is given by  
\begin{align*}
\delta_{N,2} (\boldsymbol{Z}) 
 = Y_{[2]}- \frac{1}{2} a \sigma_{yy}. 
\end{align*} 
The third estimator of $\theta_{\text{y}}^{S}$ that we propose is given by 
\begin{align*}
\delta_{N,3}  \left(  \boldsymbol{Z}\right)  = Y_{[2]} + \frac{1}{a} \ln \left[ 1 + \left( e^{ a \left( Y_{[1]}- Y_{[2]}\right) }  -1\right)  \Phi \left( \frac{X_{(1)}- X_{(2)}}{\sqrt{2\sigma_{xx}}} \right)  \right].
\end{align*}
Note that the estimator $\delta_{N,3} $ is based on the MLE of  $ \frac{1}{a}  \ln \left[ E\left(  e^{a\theta_{\text{y}}^{S}} \right)  \right]$, where $ E\left(  e^{a\theta_{\text{y}}^{S}} \right) =  e^{a\theta_y^{(2)}} \left[ 1 + \left( e^{a\left( \theta_y^{(1)}- \theta_y^{(2)} \right) } - 1  \right) \Phi \left( \frac{\theta_x^{(1)} - \theta_x^{(2)}}{\sqrt{2\sigma_{xx}}} \right) \right]$. 

\noindent Another natural estimator of $\theta_{\text{y}}^{S}$, which is similar to the estiamtor studied by Dahiya (1974),  is given by 
\begin{align*}
\delta_{N,4} \left(  \boldsymbol{Z} \right) = \left\{ \begin{array}{ll}
\frac{Y_{[1]}+Y_{[2]}}{2}, &\mbox{if}  \ \   X_{(1)} - X_{(2)}> - c\sqrt{2\sigma_{xx}}  \vspace{3mm} \\ 
Y_{[2]},  &\mbox{if} \ \   X_{(1)} - X_{(2)} \leq  - c\sqrt{2\sigma_{xx}},  
\end{array} \right. 
\end{align*}
where $c>0$ is a constant. The estimator $\delta_{N,4}$ is called  hybrid estimator and   is same as the estimator $\delta_{N,1}$ for $c=0$.
\begin{remark}
It can be verified that,  the estimator $\delta_{N,2}$ is also a generalized Bayes estimator of $\theta_{\textnormal{y}}^{S}$, using the loss function given in (\ref{loss1.1}) and  the  improper prior  
$\Pi \left(  \boldsymbol{\boldsymbol{\theta}^{(1)}, \boldsymbol{\theta}^{(2)}} \right)=1,\ \forall \ \boldsymbol{\theta}^{(i)}\in \mathbb{R}^2,\ i=1,2.$
\end{remark}
 
\begin{theorem}
	Under  the conjugate prior $ N_2 ( \boldsymbol{\mu}, \boldsymbol{\vartheta})$ and  the  loss function given in (\ref{loss1.1}), the Bayes estimator of $\theta_{\textnormal{y}}^{S}$ is given by 
	\begin{align*}\label{Bayes-est}
	\delta_{B} \left(  \boldsymbol{Z}\right)  &=    \frac{\mu_2 ( |\Sigma|+ m \sigma_{yy}) + m Y_{[2]} (m+\sigma_{xx}) +   m \sigma_{xy} (\mu_1 -X_{(2)})}  {m^2 + m \sigma_{xx} + m \sigma_{yy}+ |\Sigma| } \nonumber  \\
&  \hspace*{1.5cm} - \frac{a}{2}  \frac{m^2 \sigma_{yy} + m |\Sigma|}{  \left(  m^2 + m \sigma_{xx}+ m \sigma_{yy} +  |\Sigma| \right)}. 
	\end{align*}
\end{theorem}
\begin{proof}
Suppose that $\boldsymbol{\theta}^{(i)}$   has a conjugate  bivariate normal prior $ N_2( \boldsymbol{\mu}, \boldsymbol{\vartheta})$
where $ \boldsymbol{\mu}= \left(\mu_1, \mu_2 \right)^\prime$,  $\boldsymbol{\vartheta}=mI$, and $I$ denotes an identity matrix of order 2 and $m$ is a positive real number.  Then,  the posterior distribution of $\boldsymbol{\theta}^{(i)}$, given $\boldsymbol{Z}_i=\boldsymbol{z}_i$, is 
\begin{equation}\label{posterior}
\Pi^* \left(  \boldsymbol{\theta}^{(i)} \big| \boldsymbol{z}_i \right)  \sim  N_2  \left(  \boldsymbol{K} \left(  \boldsymbol{\Sigma}^{-1} \boldsymbol{z}_i +   \boldsymbol{\vartheta}^{-1} \boldsymbol{\mu} \right) , \boldsymbol{K}\right), \ \ \ i=1,2, 
\end{equation}
where $\boldsymbol{K} = \left(\boldsymbol{\Sigma}^{-1} + \boldsymbol{\vartheta}^{-1} \right)^{-1}$.
\\
The posterior risk of an estimator $\delta_i$  of $\theta_y^{(i)}$   under the loss function (\ref{loss1.1}) is  
\begin{equation}\label{post-exp}
E_{\Pi^*} L \left( \delta_i (\boldsymbol{Z}_i\right) , \theta_y^{(i)}) = e^{a  \delta_i \left( \boldsymbol{Z}_i\right)   } E_{\Pi^*}  \left[ e^{-a \theta_y^{(i)} } \Big| \boldsymbol{Z}_i=\boldsymbol{z}_i \right] - a \left( \delta_i \left( \boldsymbol{Z}_i \right)  - E_{\Pi^*} \left( \theta_y^{(i)} \big | \boldsymbol{Z}_i=\boldsymbol{z}_i \right) \right)  - 1, \end{equation} 
 $i=1,2$. It is not difficult to check that the Bayes estimator  $\delta_i^B (\boldsymbol{Z}_i)$ of $\theta_y^{(i)}$, which minimizes the posterior risk  (\ref{post-exp}), is given by  
\begin{equation}\label{post-minz}
  \delta_i^B (\boldsymbol{Z}_i) = -\frac{1}{a} \ln \left[  E_{\Pi^*}  \left[ e^{-a \theta_y^{(i)} } \Big| \boldsymbol{Z}_i=\boldsymbol{z}_i \right]   \right] = - \frac{1}{a}  \ln \left[   M_{\theta_y^{(i)}  \big | \boldsymbol{z}_i}  (-a ) \right],  \ \  i=1,2, 
\end{equation}
where $M_{\theta_y^{(i)} \big | \boldsymbol{z}_i} (\cdot)$ denotes the moment generating function (MGF) of $\theta_y^{(i)} \big | \boldsymbol{z}_i$. It follows from (\ref{posterior}) that   $\theta_y^{(i)} \big | \boldsymbol{z}_i $ has univariate normal distribution $N (p_i^*, q_i^*)$, where  
\begin{equation*}
p_i^*=   \frac{\mu_2 ( |\Sigma|+ m \sigma_{yy}) + m Y_i (m+\sigma_{xx}) +   m \sigma_{xy} (\mu_1 -X_i)}  {m^2 + m \sigma_{xx} + m \sigma_{yy}+ |\Sigma| },  
\end{equation*}
and 
\begin{equation*}
q_i^*=  \frac{m^2 \sigma_{yy} + m |\Sigma|}{  \left(  m^2 + m \sigma_{xx}+ m \sigma_{yy} +  |\Sigma| \right) }, \  \ \ i=1,2.
\end{equation*}
Therefore, 
\begin{equation}\label{mgf-1}
M _{\boldsymbol{\theta}^{(i)}  \big | \boldsymbol{z}_i}  (-a )  = e^{-a p_i^* + \frac{1 }{2} a^2  q_i^*}, \  \ \ i=1,2. 
\end{equation} 
Combining (\ref{post-minz}) and (\ref{mgf-1}), we get 
\begin{align*} 
\delta_i^{B} (\boldsymbol{Z}_i) & =    \frac{\mu_2 ( |\Sigma|+ m \sigma_{yy}) + m Y_i (m+\sigma_{xx}) +   m \sigma_{xy} (\mu_1 -X_i)}  {m^2 + m \sigma_{xx} + m \sigma_{yy}+ |\Sigma| }  \\
& \hspace*{1.5cm} - \frac{a}{2}  \frac{m^2 \sigma_{yy} + m |\Sigma|}{  \left(  m^2 + m \sigma_{xx}+ m \sigma_{yy} +  |\Sigma| \right) }, \ \ \   i=1,2. 
\end{align*}

\noindent It can be verified that the posterior risk  of the Bayes estimator   $ \delta_i^{B} (\boldsymbol{Z}_i)$  of $\theta_{y}^{(i)}$,  is given by
\begin{equation}\label{psot-risk}
r (\delta_i^{B}\left(  \boldsymbol{Z}_i\right) ) = \frac{a^2}{2}    \frac{  \left(  m^2 \sigma_{yy}+ |\Sigma| m \right) } {  \left(  |\Sigma| +  m^2  + m \sigma_{yy}+ m\sigma_{xx} \right) }.
\end{equation}
Since the posterior risk (\ref{psot-risk}) does not depend on  $\boldsymbol{Z}_i,\ i=1,2$, it follows form Theorem 3.1 of Sackrowitz and  Samuel-Cohen (1987) that the posterior risk $ r  \left( \delta_i^{B} \left(  \boldsymbol{Z}_i\right) \right) $, given in (\ref{psot-risk}), is also the Bayes risk  of  $ \delta_i^{B} \left( \boldsymbol{Z}_i\right)$.
Now an application of Lemma 3.2  of Sackrowitz and Samuel-Cohen (1987) leads to the result. 
\end{proof}
\begin{remark}
It can be easily checked  that the estimator  $\delta_{N,2}$ is  a limit of the Bayes  estimators $\delta_{B}\left(\boldsymbol{Z} \right) $ as $ m \to \infty$.
\end{remark}
\section{Some Admissibility Results}
In this section,  an admissible subclass of equivariant  estimators within the class $\mathcal{Q}_d$ is obtained,  using the  loss funtion given in (\ref{loss1.1}), where  
\begin{equation*}\label{sub-class}
\mathcal{Q}_d= \left\{ \delta_d : \delta_d (\boldsymbol{Z}_1,\boldsymbol{Z}_2) = Y_{[2]}+d, \  \forall \ d  \in \mathbb{R}  \right\}, 
\end{equation*}
where $\mathbb{R}$ denotes the real line. For obtaining the admissibility  of the estimators within the above class we require the following  lemma.
\begin{lemma}\label{lemma-pdf}
		Let $W=Y_{[2]}-\theta_{\textnormal{y}}^S$, and  $\rho =  \frac{\sigma_{xy}}{\sqrt{\sigma_{xx} \sigma_{yy}}}$. Then, $W$ has the   pdf   
	\begin{align*}
f_{W}(w \big|\boldsymbol{\theta}^*) =   \frac{1}{\sqrt{\sigma_{yy}} }  \phi \left( \frac{w}{\sqrt{\sigma_{yy}}} \right)  \left\{ \Phi \left( \frac{  \frac{\rho w }{\sqrt{\sigma_{yy}}}  + \frac{{\theta}_x }{\sqrt{\sigma_{xx}}} }{\sqrt{2-\rho^2}} \right)  + \Phi \left( \frac{\frac{\rho w }{\sqrt{\sigma_{yy}}} - \frac{{\theta}_x}{\sqrt{\sigma_{xx}}}}{\sqrt{2-\rho^2}} \right)  \right\}, \   \ w \in \mathbb{R}.
\end{align*}
\end{lemma}

 The following theorem establishes  the admissibility  of  the estimators  $\delta_d$ within the  class $\mathcal{Q}_d$. 
\begin{theorem}\label{thm-adm}
 Let 
$$d_0  = \left\{ \begin{array}{ll} -\frac{a\sigma_{yy}}{2}  -\frac{1}{a} \left[ \ln 2 +  \ln \left\{  \Phi \left(  \frac{a \sigma_{xy} }{\sqrt{2\sigma_{xx}}}\right)  \right\} \right],   & \textup{if} \ \ \sigma_{xy} >0 \vspace{2mm} \\
 -\frac{a\sigma_{yy}}{2}, & \textup{if} \  \ \sigma_{xy} \leq 0, 
  \end{array}\right.
   $$
   and 
$$d_1  = \left\{ \begin{array}{ll} -\frac{a\sigma_{yy}}{2},   & \textup{if} \ \ \sigma_{xy} \geq  0 \vspace{2mm} \\
   -\frac{a\sigma_{yy}}{2}  -\frac{1}{a} \left[ \ln 2 + \ln \left\{ \Phi \left(  \frac{a \sigma_{xy} }{\sqrt{2\sigma_{xx}}}\right)  \right\} \right], & \textup{if} \  \ \sigma_{xy} <0.  
\end{array}\right.
$$
 Let $\delta_d \in \mathcal{Q}_d$ be  given estimators of $\theta_{\textnormal{y}}^S$. Then, \\ (i)  Within the class $\mathcal{Q}_d$, the equivariant estimators $\delta_d$   are admissible for $d_0 \leq d \leq d_1$,    under the loss function (\ref{loss1.1}), 
\\
(ii)   The equivariant  estimators $\delta_d$ for  $  d \in \left(-\infty, d_0  \right) \cup \left( d_1, \infty \right)$  are inadmissible even within the class $\mathcal{Q}_d$.

\end{theorem}
\begin{proof}
	For a fixed $\boldsymbol{\theta}^* \in \mathbb{R}_+^2$, define  $\Psi  (\boldsymbol{\theta}^*) = - \frac{1}{a} \ln \left[ E_{\theta^*} \left(   e^{a W } \right)  \right]$, where $W=Y_{[2]}-\theta_{\text{y}}^S$. Then, 	for fixed $\boldsymbol{\theta}^* \in \mathbb{R}_+^2$, the risk function of the estimators $\delta_d$ is given by
	\begin{align*}\label{risk-dc}
	R(\delta_d, \boldsymbol{\theta}^*)=  E_{\boldsymbol{\theta}^*} \left[  e^{a \left( Y_{[2]} +d - \theta_{\text{y}}^S  \right)}  - a \left( Y_{[2]} +d  - \theta_{\text{y}}^S  \right)  - 1    \right]
	\end{align*}
	It is easy to verify that $R (\delta_d, \boldsymbol{\theta}^*)$ is minimized at  $d =  \Psi (\boldsymbol{\theta}^*) =  - \frac{1}{a} \ln \left[ E_{\boldsymbol{\theta}^*} \left(   e^{a W  } \right)  \right]$. Using Lemma \ref{lemma-pdf}, we have
		\begin{align*}
\Psi (\boldsymbol{\theta}^*) = - \frac{a \sigma_{yy}}{2} - \frac{1}{a} \ln \left[H_a(\theta_x)\right], 

		\end{align*}
		where for  $a  \neq 0, \ H_a \left( \theta_x \right) =  \Phi \left(  \frac{a \sigma_{xy} + \theta_x }{\sqrt{2\sigma_{xx}}}\right)   +  \Phi \left(  \frac{a \sigma_{xy} - \theta_x }{\sqrt{2\sigma_{xx}}}\right)$.  
		Clearly, the behaviour of $H_a (\theta_x)$ depends on $\theta_x \in (0, \infty)$. It can be verified that for $ a \sigma_{xy} >0 
		\ \left( a \sigma_{xy} <0 \right) ,$ $H_a(\theta_x)$  is a decreasing (an increasing) function of $\theta_x \in (0, \infty)$. Using the monotonicity of $ H_a \left( \theta_x \right)$,  we conclude that for $\sigma_{xy}>0 \ \left( \sigma_{xy} <0 \right) $, $\Psi \left( \boldsymbol{\theta}^* \right) $ is an increasing  (a decreasing) function of $\theta_x$. Therefore,  for $\sigma_{xy}>0$  
	\begin{equation}\label{adm-inf-sup1}
	 	\inf_{\boldsymbol{\theta}^* \in \mathbb{R}_+^2} \Psi (\boldsymbol{\theta}^*) =d_0 \ \   \text{and} \ \  \sup_{\boldsymbol{\theta}^* \in \mathbb{R}_+^2} \Psi (\boldsymbol{\theta}^*) = \lim\limits_{{\theta}_x \to \infty} \Psi(\boldsymbol{\theta}^*)=  d_1,
		\end{equation}
		and for $\sigma_{xy}<0$
			\begin{equation}\label{adm-inf-sup2}
\inf_{\boldsymbol{\theta}^* \in \mathbb{R}_+^2} \Psi (\boldsymbol{\theta}^*) = \lim\limits_{{\theta}_x \to \infty} \Psi(\boldsymbol{\theta}^*)=  d_0	\ \   \text{and} \ \ \sup_{\boldsymbol{\theta}^* \in \mathbb{R}_+^2} \Psi (\boldsymbol{\theta}^*) =  d_1.
		\end{equation}
\noindent (i) Since $ \Psi (\boldsymbol{\theta}^*)$ is a continuous function of $\boldsymbol{\theta}^*$, it follows from (\ref{adm-inf-sup1}) and (\ref{adm-inf-sup2})  that any value of $d$ in the interval $ \left(d_0, d_1 \right)$ minimizes the risk function $R(\delta_d, \boldsymbol{\theta}^*)$ for some   $\boldsymbol{\theta}^* \in  \mathbb{R}_+^2$. Consequently, the estimators $\delta_d$, for   any value of  $d \in \left(d_0, d_1 \right)$ are admissible within the subclass $\mathcal{Q}_d$. The admissibility of the estimators $\delta_{d_0}$ and $\delta_{d_1}$, within the class $\mathcal{Q}_d$, follows form continuity of $R(\delta_d, \boldsymbol{\theta}^*)$.

\noindent (ii) For a fixed  $\boldsymbol{\theta}^* \in  \mathbb{R}_+^2$, the risk function $R(\delta_d, \boldsymbol{\theta}^*)$ is a decreasing (an increasing)  function of  $d$ for  $d< \Psi (\boldsymbol{\theta}^*) \, \left(d> \Psi (\boldsymbol{\theta}^*) \right)$. Since $d_0 \leq \Psi (\boldsymbol{\theta}^*)  \leq d_1, \forall \; \boldsymbol{\theta}^* \in \mathbb{R}_{+}^{2} $, it follows that the equivariant  estimators $\delta_d $   are dominated by  $\delta_{d_0} \ \text{for} \  d< d_0$ and  $\delta_{d_1} \ \text{for} \ d> d_1$. 
\end{proof}
\begin{remark}
		The  estimator $\delta_{N,2}$  
	is a  member of the class $\mathcal{Q}_d$ for $d=-\frac{1}{2}a \sigma_{yy}$. 
	Then, using Theorem \ref{thm-adm}, the estimator $\delta_{N,2}$ is admissible within the class $\mathcal{Q}_d$. 
	\end{remark}

\section{Some  Results of Improved Estimators}
In this section, using the loss function given in (\ref{loss1.1}), a sufficient condition for improving  equivariant estimators of $\theta_{\text{y}}^S$ in the general  class $\mathcal{Q}_c$  is derived.  The following lemmas are needed for establishing the result. 
\begin{lemma}\label{cond-pdf}
	Let  $T_1 = X_{(1)} - X_{(2)}, $ $T_2 = Y_{[1]} - Y_{[2]},$ $T_3= Y_{[2]} - \theta_{\textnormal{y}}^S,$ and $\rho =  \frac{\sigma_{xy}}{\sqrt{\sigma_{xx} \sigma_{yy}}}$. 
 For $t_1 \leq 0$, $t_2 \in \mathbb{R}$, the conditional pdf of $T_3$  given $T_1=t_1, T_2=t_2$ is given by
 \newpage  
	\begin{align*}
f_{T_3|T_1,T_2}  & \left( T_3|T_1,T_2 \right)  \\ 
 & = 	\sqrt{\frac{ 2}{\sigma_{yy}}} \left[ \frac{	\phi  \left(  	\sqrt{\frac{ 2}{\sigma_{yy}}}   
		 \left(  t_3 + \frac{t_2 - \boldsymbol{\theta}_y}{2}   \right) D_1 \left(  t_1,t_2, \boldsymbol{\theta}^* \right)
		\right) + 	\phi  \left(   \sqrt{\frac{ 2}{\sigma_{yy}}} 
		\left(  t_3 + \frac{t_2 + \boldsymbol{\theta}_y}{2}   \right) D_2 \left(  t_1,t_2,\boldsymbol{\theta}^* \right)
		\right) }{D_1 \left(  t_1,t_2, \boldsymbol{\theta}^* \right) + D_2 \left(  t_1,t_2, \boldsymbol{\theta}^* \right) }
		\right],
	\end{align*}
	where 
	\begin{align*}
	D_1 \left(  t_1,t_2, {\theta}^* \right) =  \phi  \left(  \frac{ t_2 - {\theta}_{y}}{\sqrt{2\sigma_{yy}}}         \right)     \phi  \left(   \frac{\rho  \left(  \frac{ t_2 - {\theta}_{y}}{\sqrt{ \sigma_{yy}}} \right)    -  \left(  \frac{ t_1 - {\theta}_{x}}{\sqrt{ \sigma_{xx}}} \right)  }  {\sqrt{2(1-\rho^2)}} \right), 
	\end{align*}	
	and 
		\begin{align*}
	D_2 \left(  t_1,t_2, {\theta}^* \right) =  \phi  \left(  \frac{ t_2 + {\theta}_{y}}{\sqrt{2 \sigma_{yy}}}         \right)     \phi  \left(   \frac{\rho  \left(  \frac{ t_2 +  {\theta}_{y}}{\sqrt{ \sigma_{yy}}} \right)    -  \left(  \frac{ t_1 +  {\theta}_{x}}{\sqrt{2\sigma_{xx}}} \right) } {\sqrt{2(1-\rho^2)}} \right). 
	\end{align*}	
	\item [(ii)]
	  	For  $t_1 \leq 0$ and $t_2 \in \mathbb{R}$, 
\begin{equation*}
E \left( e^{aT_3 }  \big| T_1=t_1, T_2=t_2\right)  = e^{ \frac{ a^2 \sigma_{yy}}{4}   -\frac{a t_2}{2}  }  \left[ \Delta \left( t_1,t_2, \boldsymbol{\theta}^* \right)  \right],
\end{equation*}
	where for  $t_1 \leq 0$ and $t_2 \in \mathbb{R}$, 
\begin{equation}\label{Delta}
\Delta \left( t_1,t_2, \boldsymbol{\theta}^* \right) = \frac{D_1 \left(  t_1,t_2, \boldsymbol{\theta}^* \right)e^{\frac{a\boldsymbol{\theta}_{y}}{2}}+D_2 \left(  t_1,t_2, \boldsymbol{\theta}^* \right) e^{\frac{-a\boldsymbol{\theta}_{y}}{2}}}{D_1 \left(  t_1,t_2, \boldsymbol{\theta}^* \right) +D_2 \left(  t_1,t_2, \boldsymbol{\theta}^* \right)},  \ \ \  \forall \  \boldsymbol{\theta}^* \in \mathbb{R}_+^2. 
\end{equation}
%
\end{lemma}
\begin{lemma}\label{inf-sup}
	For  $t_1 \leq 0$ and $t_2 \in \mathbb{R}$, define 
	\begin{align*}
	\varphi \left(  t_1,t_2, \boldsymbol{\theta}^* \right) &  = -\frac{1}{a} \ln \left[  E \left( e^{aT_3}  \big| T_1=t_1, T_2=t_2\right)  \right] \\ 
	& = \frac{t_2}{2} -\frac{a \sigma_{yy}}{4}   - \frac{1}{a} \ln   \left[ \Delta \left( t_1,t_2,\boldsymbol{\theta}^* \right) \right] \ \ \text{(Using Lemma \ref{cond-pdf} (ii))},
	\end{align*}
where $\Delta (\cdot)$ is given by  (\ref{Delta}). Then, for  $t_1 \leq 0$ and $t_2 \in \mathbb{R}$, 
\begin{equation*}
\varphi_I \left(  t_1,t_2  \right) \leq  \varphi \left(  t_1,t_2, \boldsymbol{\theta}^* \right) \leq  \varphi_S \left(  t_1,t_2 \right), \ \  \forall \, \boldsymbol{\theta}^* \in \mathbb{R}_+^2, 
\end{equation*}
where 
	\begin{align*}
\varphi_{I}  \left(  t_1,t_2  \right) &= \left\{   \begin{array}{ll} 
\frac{t_2}{2}- \frac{a \sigma_{yy}}{4}, & \textup{if} \  \  t_1 \xi -  \rho t_2<0 \ \text{and} \   t_2- \xi \rho t_1  <- a \frac{\sigma_{yy}}{2}  (1-\rho^2)
\vspace{3mm}	\\ 
-\infty, &  \textup{otherwise}, 
\end{array}
\right.  
\end{align*}			
and
\begin{align*}
\varphi_{S}  \left(  t_1,t_2  \right)
&= \left\{   \begin{array}{ll} 
\frac{t_2}{2}- \frac{a \sigma_{yy}}{4}, &\textup{if} \ \   t_1 \xi -  \rho t_2 > 0 \ \text{and} \  t_2 - \xi \rho t_1 > -a \frac{\sigma_{yy}}{2}  (1-\rho^2)
\vspace{3mm} \\ 
\infty, &  \textup{otherwise}, 
\end{array}
\right.
\end{align*}
	where $  \xi=\sqrt{\frac{\sigma_{yy}}{\sigma_{xx}}}$. 
	\end{lemma}
Now,  
 we exploit the approach of Brewster and Zidek (1974)  to obtain a sufficient condition for improving the equivariant estimators of the form $ \delta_{\varphi} \left( \boldsymbol{Z} \right) = Y_{[2]}+ \varphi \left( T_1,T_2 \right) $, where  $T_1 = X_{(1)} - X_{(2)} $ and $T_2 = Y_{[1]} - Y_{[2]}.$

\begin{theorem}\label{thm-s-c}
	Consider an equivariant estimator $\delta_{\varphi} \left( \boldsymbol{Z} \right)  = Y_{[2]} + \varphi \left( T_1,T_2 \right) $ of $\theta_{\textnormal{y}}^S$, where $\varphi (\cdot)$ denotes a function of $T_1$ and $T_2$. 
 Suppose that  $P  \left(
  \left\{ \varphi(T_1, T_2) \leq  \varphi_{I}(T_1, T_2)  \right\} \right.$  $\left.   \cup
 \left\{ \varphi(T_1, T_2) \geq  \varphi_{S}(T_1, T_2)  \right\} \right)  >0$, where $\varphi_{I}(\cdot)$ and $\varphi_{S}(\cdot)$  are as given in Lemma \ref{inf-sup}. Then, using the loss function given in (\ref{loss1.1}), the estimator $\delta_{\varphi} (\cdot)$  is improved by 
$\delta_{\varphi}^* (\boldsymbol{Z})=Y_{[2]} + \varphi^*(T_1, T_2)$, where 
\begin{equation*}
\varphi^*(T_1, T_2)=  \left\{ \begin{array}{ll} 
 \varphi_{I} (T_1, T_2),   & \mbox{if}  \ \   \varphi(T_1, T_2) \leq   \varphi_{I}(T_1, T_2) \vspace{3mm} \\
  \varphi (T_1, T_2) , & \mbox{if}  \ \ \varphi_{I}(T_1, T_2) <  \varphi(T_1, T_2) <  \varphi_{S}(T_1, T_2) \vspace{3mm} \\ 
  \varphi_{S}(T_1, T_2),  & \mbox{if} \ \ \varphi(T_1, T_2) \geq \varphi_{S}(T_1, T_2).
 \end{array} \right.
\end{equation*} 
\end{theorem}
\begin{proof}
	(i) \ 	Consider the risk difference of the estimators $\delta_\varphi$ and $\delta_\varphi^*$ and 
	\begin{equation*}
	R(\boldsymbol{\theta}^*, \delta_\varphi)-R(\boldsymbol{\theta}^*, \delta_\varphi^*)= E \left[ K_{\boldsymbol{\theta}^*} (T_1, T_2) \right], 
	\end{equation*}
	where, for $t_1 \leq 0, \ t_2 \in \mathbb{R}$,
		\begin{align*}
	K_{\boldsymbol{\theta}^*} (t_1,t_2) &=  E \left[  e^{a\left( \delta_{\varphi} (\boldsymbol{Z})-\theta_{\text{y}}^S\right)  }-a(\delta_{\varphi}(\boldsymbol{Z})-\theta_{\text{y}}^S)-1 \Big|\, T_1=t_1, T_2=t_2 \right] \vspace{4mm} \\
	& \hspace*{3cm} -  E \left[  e^{a\left( \delta_{\varphi}^*(\boldsymbol{Z})-\theta_{\text{y}}^S\right) }-a \left( \delta_{\varphi}^* (\boldsymbol{Z})-\theta_{\text{y}}^S \right) -1  \Big| \, T_1=t_1, T_2=t_2 \right] \vspace{4mm} \\
	& =  E \left[  e^{a\left( \delta_{\varphi}  (\boldsymbol{Z})-\theta_{\text{y}}^S \right) }-e^{a\left( \delta_{\varphi}^* (\boldsymbol{Z})-\theta_{\text{y}}^S \right) } \Big| T_1=t_1, T_2=t_2 \right] \vspace{4mm} \\ 
	& \hspace*{3cm} -aE \left[ \delta_{\varphi} (\boldsymbol{Z})-\delta_{\varphi}^* (\boldsymbol{Z}) \Big| T_1=t_1, T_2=t_2 \right] \vspace{4mm} \\
	& = E \left[  e^{a \left( Y_{[2]}+\varphi(t_1,t_2)-\theta_{\text{y}}^S \right) }-e^{a\left( Y_{[2]}+\varphi^* (t_1,t_2)-\theta_{\text{y}}^S \right) } \big| T_1=t_1, T_2=t_2 \big]-a\big[ \varphi(t_1,t_2)-\varphi^* (t_1,t_2)\right] \\
	& =  \left[ e^{a\varphi(t_1,t_2)}-e^{ a \varphi^*(t_1,t_2)} \right]  E \left(e^{a\left( Y_{[2]} - \theta_{\text{y}}^S\right) } \big| T_1=t_1, T_2=t_2 \right)  -  a\left[\varphi(t_1,t_2)-\varphi^* (t_1,t_2)\right]\\
	& =  \left[ e^{a\varphi(t_1,t_2)}-e^{a\varphi^* \left( t_1,t_2\right) } \right] e^{-a\varphi(t_1,t_2,\boldsymbol{\theta}^*)} - a\left[\varphi\left( t_1,t_2\right) -\varphi^* \left( t_1,t_2\right) \right].
	\end{align*}
The last line of the above expression follows from  Lemma \ref{cond-pdf} and  Lemma \ref{inf-sup}. 
 Now,  for a fixed $t_1 \leq 0$ and $t_2 \in \mathbb{R},$  if  $\varphi(t_1,t_2)\leq \varphi_{I}(t_1,t_2)  \left( \text{so that} \, \varphi^* (t_1,t_2)=\varphi_{I}(t_1,t_2) \right) $, then,  
	\begin{align*}
	K_{\boldsymbol{\theta}^*}(t_1,t_2)  &=  \left[ e^{a\varphi(t_1,t_2)}-e^{a\varphi_{I}(t_1,t_2)} \right] e^{-a\varphi(t_1,t_2,\boldsymbol{\theta}^*)} - a\left(\varphi(t_1,t_2)-\varphi_{I}(t_1,t_2)\right) \\
	&\geq \left[ e^{a\varphi(t_1,t_2)}-e^{a\varphi_{I}(t_1,t_2)} \right] e^{-a\varphi_I (t_1,t_2)} - a\left(\varphi(t_1,t_2)-\varphi_{I}(t_1,t_2)\right) \\
	&= \left[ e^{a\{\varphi(t_1,t_2)-\varphi_{I}(t_1,t_2)\}} -1\right]-a\left[\varphi(t_1,t_2)-\varphi_{I}(t_1,t_2)\right].
	\end{align*}
	Using the property  $e^x > 1+x, \, \forall \, x \neq 0$, we have  $K_{\boldsymbol{\theta}^*}  \left( t_1,t_2 \right)  \geq 0$.  If  $\varphi_I(t_1,t_2) <  \varphi(t_1,t_2) < \varphi_{S}(t_1,t_2) \left( \text{so that} \varphi^* (t_1,t_2)=\varphi(t_1,t_2)\right) $, then,  $K_{\boldsymbol{\theta}^*}(t_1,t_2)=0$.
 If  $\varphi (t_1,t_2) \geq \varphi_{S}(t_1,t_2)  \left( \text{so that} \, \varphi^* (t_1,t_2)=\varphi_{S}(t_1,t_2) \right) $, then, 
 	\begin{align*}
 K_{\boldsymbol{\theta}^*}(t_1,t_2)  &=  \left[ e^{a\varphi(t_1,t_2)}-e^{a\varphi_{S}(t_1,t_2)} \right] e^{-a\varphi(t_1,t_2,\boldsymbol{\theta}^*)} - a\left(\varphi(t_1,t_2)-\varphi_{S}(t_1,t_2)\right) \\
 &\geq \left[ e^{a\varphi(t_1,t_2)}-e^{a\varphi_{S}(t_1,t_2)} \right] e^{-a\varphi_I (t_1,t_2)} - a\left(\varphi(t_1,t_2)-\varphi_{S}(t_1,t_2)\right) \\
 &= \left[ e^{a\{\varphi(t_1,t_2)-\varphi_{S}(t_1,t_2)\}} -1\right]-a\left[\varphi(t_1,t_2)-\varphi_{S}(t_1,t_2)\right].
 \end{align*}
		Again using the property  $e^x > 1+x, \, \forall \, x \neq 0$, we have  $K_{\boldsymbol{\theta}^*}  \left( t_1,t_2 \right)  \geq 0$.
\noindent	Now, since  $P \left(
	\left\{ \varphi(T_1, T_2) \leq  \varphi_{I}(T_1, T_2)  \right\}  \right.$ $ \left. \cup
	\left\{ \varphi(T_1, T_2) \geq  \varphi_{S}(T_1, T_2)  \right\} \right)  >0$, we conclude that 
	\begin{equation*}
	R(\boldsymbol{\theta}^*, \delta_{\varphi})-R(\boldsymbol{\theta}^*, \delta_{\varphi}^*) \geq 0,  \ \ \forall \,  \boldsymbol{\theta}^*  \in  \mathbb{R}_+^2,
	\end{equation*}
and the srtict inequality holds for some $ \boldsymbol{\theta}^*  \in  \mathbb{R}_+^2.$
Hence the result follows. 
\end{proof}

\section*{Improved Estimators}
	Here,  we provide some improved estimators of  $\theta_{\text{y}}^S$ by using the reuslt of Theorem \ref{thm-s-c}.  

\noindent	  \textbf{Improved estimator 1:} 
		For  $a>0$ and  $ 0 < \rho \leq 1$,  the estimator $\delta_{N,1}$  is improved by
		\begin{align*}
		\delta_{N,1}^{I1}  \left(  \boldsymbol{Z}_1,\boldsymbol{Z}_2  \right)
		&= \left\{   \begin{array}{ll} 
		\frac{Y_{[1]}+Y_{[2]} }{2}- \frac{a \sigma_{yy}}{4}, & \textup{if}   \ \ T_1 > \frac{\rho T_2}{\xi} \ \textup{and}  \      \frac{a \sigma_{yy}}{2}  \geq  T_2  >\xi \rho T_1 -a \frac{\sigma_{yy}}{2} (1-\rho^2)
		\vspace{3mm} \\ 
		\delta_{N,1}, &  \textup{otherwise}. 
		\end{array}
		\right.
		\end{align*}
	
\noindent  \textbf{Improved estimator 2:} 
For  $a<0$ and  $ -1 \leq \rho <0$,  the estimator $\delta_{N,1}$  is improved by
\begin{align*}
\delta_{N,1}^{I2}  \left(  \boldsymbol{Z}_1,\boldsymbol{Z}_2  \right)
&= \left\{   \begin{array}{ll} 
\frac{Y_{[1]}+Y_{[2]} }{2}- \frac{a \sigma_{yy}}{4}, & \textup{if}   \ \ T_1 < \frac{\rho T_2}{\xi} \ \textup{and}  \  \frac{a \sigma_{yy}}{2}    \leq T_2  <  \xi \rho T_1 -a \frac{\sigma_{yy}}{2} (1-\rho^2)  
\vspace{3mm} \\ 
\delta_{N,1}, &  \textup{otherwise}. 
\end{array}
\right.
\end{align*}		

	\noindent \textbf{Improved estimator 3:} 
For  $a>0$ $(a<0)$ and  $-1  \leq   \rho < 0$ $(0< \rho \leq1)$,  the estimator $\delta_{N,1}$  is improved by  
\begin{align*}
\delta_{N,1}^{I3}  \left(  \boldsymbol{Z}_1,\boldsymbol{Z}_2  \right)
&= \left\{   \begin{array}{ll} 
\frac{Y_{[1]}+Y_{[2]} }{2}- \frac{a \sigma_{yy}}{4}, & \textup{if}   \ \ T_1<\frac{\rho T_2}{\xi} \ \textup{and}  \   \frac{a \sigma_{yy}}{2}   \leq T_2  <    \xi \rho T_1 -a \frac{\sigma_{yy}}{2} (1-\rho^2) \vspace{3mm} \\
& \textup{or}  \ \ T_1>\frac{\rho T_2}{\xi} \ \textup{and}  \        \frac{a \sigma_{yy}}{2}  \geq  T_2  > \xi \rho T_1 -a \frac{\sigma_{yy}}{2}(1-\rho^2)    
\vspace{3mm} \\ 
\delta_{N,1}, &  \textup{otherwise}. 
\end{array}
\right.
\end{align*}
\noindent \textbf{Improved estimator 4:} 
For  $a<0$ and  $  \rho = 0$,  the estimator $\delta_{N,1}$  is improved by
\begin{align*}
\delta_{N,1}^{I4}  \left(  \boldsymbol{Z}_1,\boldsymbol{Z}_2  \right)
&= \left\{   \begin{array}{ll} 
\frac{Y_{[1]}+Y_{[2]} }{2}- \frac{a \sigma_{yy}}{4}, & \textup{if}   \ \  \frac{a \sigma_{yy}}{2}   \leq T_2  <  -\frac{a \sigma_{yy}}{2} 
\vspace{3mm} \\ 
\delta_{N,1}, &  \textup{otherwise}. 
\end{array}
\right.
\end{align*}
For $a>0$  and $\rho =0$, Theorem \ref{thm-s-c} fails to provide an improved estimator upon the estimator $\delta_{N,1}$.

\noindent \textbf{Improved estimator 5:} 
		For  $a>0$ and  $-1  \leq \rho < 0$,  the estimator $\delta_{N,2}$  is improved by
		\begin{align*}
		\delta_{N,2}^{I1}  \left(  \boldsymbol{Z}_1,\boldsymbol{Z}_2  \right)
		&= \left\{   \begin{array}{ll} 
		\frac{Y_{[1]}+Y_{[2]} }{2}- \frac{a \sigma_{yy}}{4}, & \textup{if}   \ \ T_1<\frac{\rho T_2}{\xi} \ \textup{and}  \   -\frac{a \sigma_{yy}}{2}   \leq T_2  <    \xi \rho T_1 -a \frac{\sigma_{yy}}{2} (1-\rho^2)
		\vspace{3mm} \\ 
		\delta_{N,2}, &  \textup{otherwise}. 
		\end{array}
		\right.
		\end{align*}	
\noindent	 \textbf{Improved estimator 6:} 
For  $a > 0 $ $(a <0)$ and  $0 <  \rho \leq 1$ $(-1 \leq \rho < 0)$,  the estimator $\delta_{N,2}$  is improved by  
\begin{align*}
\delta_{N,2}^{I2}  \left(  \boldsymbol{Z}_1,\boldsymbol{Z}_2  \right)
&= \left\{   \begin{array}{ll} 
\frac{Y_{[1]}+Y_{[2]} }{2}- \frac{a \sigma_{yy}}{4}, & \textup{if}   \ \ T_1<\frac{\rho T_2}{\xi} \ \textup{and}  \   -\frac{a \sigma_{yy}}{2}   \leq T_2  <    \xi \rho T_1 -\frac{a\sigma_{yy}}{2} (1-\rho^2) \vspace{3mm} \\
   & \textup{or}  \ \ T_1>\frac{\rho T_2}{\xi} \ \textup{and}  \  -\frac{a \sigma_{yy}}{2}  \geq T_2> \xi \rho T_1 -a \frac{\sigma_{yy}}{2}(1-\rho^2)         
\vspace{3mm} \\ 
\delta_{N,2}, &  \textup{otherwise}. 
\end{array}
\right.
\end{align*}
	For $a<0 \ (a\neq 0)$ and  $0 < \rho  \leq 1  \ (\rho =0)$, Theorem \ref{thm-s-c} fails to provide an improved estimator upon the  estimator $\delta_{N,2}$.
	
	\noindent \textbf{Improved estimator 7:} 
		For  $a>0$,  $0  <\rho \leq  1$, and $\varphi_3\leq \varphi_{I}$ or $\varphi_3\geq \varphi_{S}$, where $\varphi_3=\frac{1}{a} \ln \left[ 1 + \left( e^{ a T_2 }  -1\right)  \Phi \left( \frac{T_1}{\sqrt{2\sigma_{xx}}} \right)  \right]$, and $\varphi_{I}$ and $\varphi_{S}$ are as given in Lemma \ref{inf-sup},  the  estimator $\delta_{N,3}$  is improved by
		\begin{align*}
		\delta_{N,3}^{I1}  \left(  \boldsymbol{Z}_1,\boldsymbol{Z}_2  \right)
		&= \left\{   \begin{array}{ll} 
		\frac{Y_{[1]}+Y_{[2]} }{2}- \frac{a \sigma_{yy}}{4}, & \textup{if}  \ \  T_1   <  \frac{\rho T_2}{ \xi} \ \  \textup{and} \ \   T_2  <    \xi \rho T_1 -a \frac{\sigma_{yy}}{2} (1-\rho^2)
		\vspace{3mm} \\
		& \textup{or}  \ \  T_1   >  \frac{\rho T_2}{ \xi} \ \ \textup{and}  \ \ T_2> \xi \rho T_1 -a \frac{\sigma_{yy}}{2} (1-\rho^2)      
	\vspace{3mm} \\	 
		\delta_{N,3}, &  \textup{otherwise}. 
		\end{array}
		\right.
		\end{align*}	
	\noindent  \textbf{Improved estimator 8:} 
		For  $a<0$ and  $0  <\rho \leq  1$ and $\varphi_3\leq \varphi_{I}$,  the estimator $\delta_{N,3}$  is improved by
		\begin{align*}
		\delta_{N,3}^{I2}  \left(  \boldsymbol{Z}_1,\boldsymbol{Z}_2  \right)
		&= \left\{   \begin{array}{ll} 
			\frac{Y_{[1]}+Y_{[2]} }{2}- \frac{a \sigma_{yy}}{4}, & \textup{if}   \ T_1   <  \frac{\rho T_2}{ \xi} \ \  \textup{and} \ \   T_2  <    \xi \rho T_1 -a \frac{\sigma_{yy}}{2} (1-\rho^2)  
		\vspace{3mm} \\	 
		\delta_{N,3}, &  \textup{otherwise}. 
		\end{array}
		\right.
		\end{align*}	
\noindent  \textbf{Improved estimator 9:} 
		For  $a \neq 0$, $ -1 \leq  \rho < 0$ and $\varphi_3\leq \varphi_{I}$ or $\varphi_3\geq \varphi_{I}$,  the  estimator $\delta_{N,3}$  is improved by
		\begin{align*}
		\delta_{N,3}^{I3}  \left(  \boldsymbol{Z}_1,\boldsymbol{Z}_2  \right)
		&= \left\{   \begin{array}{ll} 
		\frac{Y_{[1]}+Y_{[2]} }{2}- \frac{a \sigma_{yy}}{4}, & \textup{if}  \ \   T_2  < \min   \left\{ \frac{\xi T_1}{\rho },  \xi \rho T_1 -a \frac{\sigma_{yy}}{2} (1-\rho^2) \right\}  
		\vspace{3mm} \\
		 & \textup{or}  \ \  \max \left\{ \frac{\xi T_1}{\rho },  \xi \rho T_1 -a \frac{\sigma_{yy}}{2} (1-\rho^2) \right\} < T_2    
		\vspace{3mm} \\	 
		\delta_{N,3}, &  \textup{otherwise}. 
		\end{array}
		\right.
		\end{align*}	
		\noindent \textbf{Improved estimator 10:} 
		For  $a \neq 0$,  $  \rho = 0$ and $\varphi_3\leq \varphi_{I}$,  the estimator $\delta_{N,3}$  is improved by
		\begin{align*}
		\delta_{N,3}^{I4}  \left(  \boldsymbol{Z}_1,\boldsymbol{Z}_2  \right)
		&= \left\{   \begin{array}{ll} 
	 	\frac{Y_{[1]}+Y_{[2]} }{2}- \frac{a \sigma_{yy}}{4} & \textup{if}  \ \  T_1< 0 \  \textup{and}  \   T_2  < - \frac{a\sigma_{yy}}{2} 
		\vspace{3mm} \\
		\delta_{N,3}, &  \textup{otherwise}. 
		\end{array}
		\right.
		\end{align*}	
\noindent \textbf{Improved estimator 11:}  	
				For $a>0$ and  $0 < \rho \leq 1$,  the estimator $\delta_{N,4}$  is improved by  
			\begin{align*}
			\delta_{N,4}^{I1}  \left(  \boldsymbol{Z}_1,\boldsymbol{Z}_2  \right)
			&= \left\{   \begin{array}{ll} 
			\frac{Y_{[1]}+Y_{[2]}}{2}- \frac{a \sigma_{yy}}{4}, & \textup{if} \   T_1  > \max \left\{- c\sqrt{2\sigma_{xx}}, \frac{\rho T_2}{\xi}\right\}  \ \text{and}  \  T_2 > \xi \rho T_1 - \frac{a\sigma_{yy}}{2}  (1-\rho^2)  
			\vspace{3mm} \\ 
				\frac{Y_{[1]}+Y_{[2]}}{2}- \frac{a \sigma_{yy}}{4}, & \textup{if} \  \frac{\rho T_2}{\xi} < T_1 \leq - c\sqrt{2\sigma_{xx}} \  \text{and}  \ \  \frac{a\sigma_{yy}}{2} \geq T_2 > \xi \rho T_1 -\frac{a\sigma_{yy}}{2}  (1-\rho^2)   
			\vspace{3mm} \\ 
			\delta_{N,4}, &  \textup{otherwise}. 
			\end{array}
			\right.
			\end{align*}
	\noindent \textbf{Improved estimator 12:}  	
For $a>0$ and  $-1\leq \rho < 0$,  the estimator $\delta_{N,4}$  is improved by  
\begin{gather}
\scalebox{1}{$  \begin{align*}
\delta_{N,4}^{I2}  \left(  \boldsymbol{Z}_1,\boldsymbol{Z}_2  \right)
&= \left\{   \begin{array}{ll} 
\frac{Y_{[1]}+Y_{[2]}}{2}- \frac{a \sigma_{yy}}{4}, & \textup{if} \ T_1 > \max \left\{- c\sqrt{2\sigma_{xx}}, \frac{\rho T_2}{\xi}\right\}   \ \text{and}  \  T_2 >  \xi \rho T_1 - \frac{a \sigma_{yy}}{2}  (1-\rho^2)   
\vspace{3mm} \\ 
\frac{Y_{[1]}+Y_{[2]}}{2}- \frac{a \sigma_{yy}}{4}, & \textup{if} \ T_1< \min \left\{- c\sqrt{2\sigma_{xx}}, \frac{\rho T_2}{\xi}\right\}  \  \text{and}  \ \frac{a\sigma_{yy}}{2} \leq T_2 < \xi \rho T_1 - \frac{a\sigma_{yy}}{2}  (1-\rho^2)   \vspace{3mm}  \\
& \textup{or} \  \frac{\rho T_2}{\xi} < T_1 \leq - c\sqrt{2\sigma_{xx}} \  \text{and}  \ \frac{a\sigma_{yy}}{2} \geq T_2 >  \xi \rho T_1 - \frac{a\sigma_{yy}}{2}  (1-\rho^2)   
\vspace{3mm} \\ 
\delta_{N,4}, &  \textup{otherwise}. 
\end{array}
\right.\end{align*}$}
\end{gather}			
\noindent\textbf{Improved estimator 13:}  	
For $a<0$ and  $0 < \rho \leq 1$,  the estimator $\delta_{N,4}$  is improved by  
\begin{gather}
\scalebox{1}{$ \begin{align*}
	\delta_{N,4}^{I3}  \left(  \boldsymbol{Z}_1,\boldsymbol{Z}_2  \right)
	&= \left\{   \begin{array}{ll} 
	\frac{Y_{[1]}+Y_{[2]}}{2}- \frac{a \sigma_{yy}}{4}, & \textup{if} \  - c\sqrt{2\sigma_{xx}} <T_1 <\frac{\rho T_2}{\xi}   \ \text{and}  \  T_2 <  \xi \rho T_1 - \frac{ a \sigma_{yy}}{2}  (1-\rho^2)  
	\vspace{3mm} \\ 
	\frac{Y_{[1]}+Y_{[2]}}{2}- \frac{a \sigma_{yy}}{4}, & \textup{if} \ T_1< \min \left\{- c\sqrt{2\sigma_{xx}}, \frac{\rho T_2}{\xi}\right\}  \  \text{and}  \ \frac{a\sigma_{yy}}{2} \leq T_2 < \xi \rho T_1 -\frac{a \sigma_{yy}}{2}  (1-\rho^2)   \vspace{3mm}  \\
	& \textup{or} \  \frac{\rho T_2}{\xi} < T_1 \leq - c\sqrt{2\sigma_{xx}} \  \text{and}  \  \frac{a\sigma_{yy}}{2} \geq T_2 > \xi \rho T_1 - \frac{a\sigma_{yy}}{2}  (1-\rho^2)  
	\vspace{3mm} \\ 
	\delta_{N,4}, &  \textup{otherwise}. 
	\end{array}
	\right.\end{align*}$}
\end{gather}			
\noindent \textbf{Improved estimator 14:}  			
For $a<0$ and  $ -1 \leq \rho < 0$,  the estimator $\delta_{N,4}$  is improved by  
\begin{gather}
\scalebox{.95}{$ \begin{align*}
	\delta_{N,4}^{I4}  \left(  \boldsymbol{Z}_1,\boldsymbol{Z}_2  \right)
	&= \left\{   \begin{array}{ll} 
	\frac{Y_{[1]}+Y_{[2]}}{2}- \frac{a \sigma_{yy}}{4}, & \textup{if} \  - c\sqrt{2\sigma_{xx}} <T_1 <\frac{\rho T_2}{\xi}   \ \text{and}  \  T_2 <  \xi \rho T_1 - \frac{ a \sigma_{yy}}{2}  (1-\rho^2)  
	\vspace{3mm} \\ 
	\frac{Y_{[1]}+Y_{[2]}}{2}- \frac{a \sigma_{yy}}{4}, & \textup{if} \ T_1< \min \left\{- c\sqrt{2\sigma_{xx}}, \frac{\rho T_2}{\xi}\right\}  \  \text{and}  \ \frac{a\sigma_{yy}}{2} \leq  T_2 < \xi \rho T_1 -a \frac{\sigma_{yy}}{2}  (1-\rho^2) 
	  \vspace{3mm}  \\
	\delta_{N,4}, &  \textup{otherwise}. 
	\end{array}
	\right.\end{align*}$}
\end{gather}			
\noindent  \textbf{Improved estimator 15:}  			
For $a<0$ and  $  \rho = 0$,  the estimator $\delta_{N,4}$  is improved by  
 \begin{align*}
	\delta_{N,4}^{I5}  \left(  \boldsymbol{Z}_1,\boldsymbol{Z}_2  \right)
	&= \left\{   \begin{array}{ll} 
	\frac{Y_{[1]}+Y_{[2]}}{2}- \frac{a \sigma_{yy}}{4}, & \textup{if} \   T_1> - c\sqrt{2\sigma_{xx}}   \ \text{and}  \  T_2 <   - \frac{a\sigma_{yy}}{2}  
	\vspace{3mm} \\ 
	\frac{Y_{[1]}+Y_{[2]}}{2}- \frac{a \sigma_{yy}}{4}, & \textup{if} \ T_1\leq- c\sqrt{2\sigma_{xx}} \  \text{and}  \ \frac{a\sigma_{yy}}{2} \leq  T_2 <  - \frac{a\sigma_{yy}}{2}  
	\vspace{3mm}  \\
	\delta_{N,4}, &  \textup{otherwise}. 
	\end{array}
	\right.\end{align*}	
	\\		
		For $a>0$ and  $\rho =0$, Theorem \ref{thm-s-c} fails to provide an improved estimator upon the estimator $\delta_{N,4}$.

\section{An application to Poultry feeds data }
In this section, a data analysis is presented using a real data set (reported in Olosunde (2013)) to domenstrate the computation of various estimates of $\theta_{\text{y}}^S$.   
Olosunde (2013) conducted a study to compare the effect of two different  copper-salt combinations on eggs produced by chicken in  poultry feeds. An equal number of chickens were randomly assigned to be fed with each of the two combinations.  A sample of 96 chickens were randomly selected from the poultry and were divided into two groups, of 48 chickens each.  One group was given an organic copper-salt combination and an inorganic copper-salt combination was given to the another group.  After a period of time, the weight  and the cholesterol level  of the eggs produced  by the two groups were measured.  The observed data from the organic  and the inorganic  Copper-Salt combinations are  reported in Olosunde (2013) and presented in Table \ref{Data}.  The eggs with more weights and less cholesterol is preferable.


Let $\pi_1$ and $\pi_2$ represent the populations given  an  organic copper-salt combination and an  inorganic  copper-salt combination, respectively.  
  Let $(X_i, Y_i)$ be a pair of  observations from the population $\pi_i,  \, i=1,2,$ where the $X$-variate  denotes the average  weights of  eggs and the $Y$-variate denotes the corresponding average cholesterol  levels. A number of 48 observations corresponding to each measurement is available from the data obtained by Olosunde (2013). Since the sample sizes of the two populations are same, the pooled variance-covaraince matrix is used.  The obtained data are assumed to  have a bivariate normal distribution with different means and common known  variance-covariance matrix.  To check the validity of the bivariate normality assumption for the available data set, we apply the Royston's normality test, given in the R-software package “MVN" that provided by Korkmaz et al. 2014. Royston's test combines the Shapiro-Wilk (S-W) test statistics for univaraite normality and  obtain one test statistic for bivariate/multivariate normality. The  Royston's  and Shapiro-Wilk tests statistic  with corresponding p-values are presented in Table \ref{royston-test}. 
\begin{table}[H] 	
	\begin{center}
		\caption{Normality  test, p-values,  kurtosis and skewness.} 
		\vspace{.1cm} 
	\label{royston-test}	\setlength{\tabcolsep}{6pt}
		\def\arraystretch{1.2}
		\begin{tabular}{|ccccccc|}
			\hline
			\cline{1-7} 
		Test& Measure & Statistic &p-value &kurtosis &   Skewness  &Normality \\
			\hline
	Royston 		& {$\pi_1$} &5.878109  &0.0529& && Yes
			\\
S-W		& {$\pi_1$}-weight  &0.9569  &  0.0758 &-1.256476&  0.01487668  & Yes  \\ 
S-W		&{$\pi_1$}-cholesterol  &0.9598  &  0.0988   &-1.213823& -0.09288089& Yes \\	
	Royston 	 &{$\pi_2$}  & 2.867&0.1051&  && Yes
\\
S-W	&	{$\pi_2$}-weight  &0.9679 &0.2089    &-1.110509 & 0.1015675& Yes\\
S-W		&{$\pi_2$}-cholesterol &0.9543 &0.0592 &-1.263555 &  -0.0816612 & Yes	\\	 
			\hline 			
		\end{tabular}    
	\end{center}
\end{table}
 From Table \ref{royston-test}, we may  conclude that the data set satisfy the  bivariate normality assumption at 0.05 level of significance.    The estimated parameters of the bivariate normal model (based on ML) are presented in Table \ref{table-data}. 
\begin{table}[H] 	
	\begin{center}
		\caption{Estimated parameters of the bivariate normal distribution.} 
		\vspace{.1cm} 
		\label{table-data}	\setlength{\tabcolsep}{7pt}
		\def\arraystretch{1.2}
		\begin{tabular}{|ccccc|}
			\hline
			\cline{1-5} 
			Population& Measure& Mean& Variance &Covariance\\
			\hline
			{$\pi_1$} &weight 	&59.0997& 8.1645& 40.0655 \\ 
			& cholesterol &131.4569&952.9425&\\	
			{$\pi_2$}	&weight	 &58.3516&8.1645& 40.0655\\
			& cholesterol &195.7275&952.9425&\\	
			\hline 			
		\end{tabular}    
	\end{center}
\end{table}
Recall that, the quality of a population is determined with regard to their X-variate, while the corresponding Y-variate is of main interest.  
We say that the population $\pi_1 \equiv N \left( \boldsymbol{\theta}^{(1)},   \boldsymbol{\Sigma}  \right)$ is  better than the population  $\pi_2 \equiv N \left(  \boldsymbol{\theta}^{(2)},  \boldsymbol{\Sigma}  \right)$ if 
$\theta_x^{(1)} > \theta_x^{(2)}$ and the population  $\pi_2$ is considered  better than the population $\pi_1$  if $\theta_x^{(1)} \leq \theta_x^{(2)}$, where 
$ \boldsymbol{\theta}^{(1)}= \left(\theta_x^{(1)}, \theta_y^{(1)} \right)^\intercal$ and  $ \boldsymbol{\theta}^{(2)}= \left(\theta_x^{(2)}, \theta_y^{(2)} \right)^\intercal$ are the mean vectors of the populations $\pi_1$ and $\pi_2$   respectively. From the data  we have $ \boldsymbol{\hat{\theta}}^{(1)} = \left( 59.0998, 131.4569 \right)^\intercal$, $ \boldsymbol{\hat{\theta}}^{(2)} = \left(  58.3517,  195.7275 \right)^\intercal$, and  $ \boldsymbol{\Sigma} = \begin{bmatrix}
8.1645 &  40.0655  \\
 40.0655 & 952.9425
\end{bmatrix}$.  
 It can be  observed that the average weight of eggs from chicken fed with an organic copper-salt combination is larger than the one with an in-organic copper-salt combination. 
Therefore, using the natural selection rule $\boldsymbol{\psi}$ given in  (\ref{sel-rule}), we may conclude that the population $\pi_1$ is preferable over   the population $\pi_2$. Also, the average cholesterol level for  the population $\pi_1$  is less than that for the   population $\pi_2$. Hence, based on the above observations,  the organic copper-salt combination is recommended. This result was  also  obtained by Olosunde (2013).  The various estimates of $\theta_{\textnormal{y}}^S$ of the selected bivariate normal population are presented in Tables \ref{table-est1} and \ref{table-est2}.
\begin{table}[H]
	\begin{center}
		\caption{The various estimates of $\theta_{\textnormal{y}}^S$ for $a=1$.} 
		\vspace{.1cm} 
		\label{table-est1} 	\setlength{\tabcolsep}{6pt}
		\def\arraystretch{1.4}
		\begin{tabular}{|c|c|c|c|c|c|c|c|}
			\hline
		$\delta_{N,1}$& $\delta_{N,1}^{I1}$ & 	$\delta_{N,2}$& 	$\delta_{N,2}^{I2}$& 	$\delta_{N,3}$ & 	$\delta_{N,3}^{I1}$ & 	$\delta_{N,4}$ & $\delta_{N,4}^{I1}$\\
			\cline{1-8} 
		131.4569	&131.4569&-345.0144 &-345.0144&194.9654&194.9654&163.5922& 163.5922 \\
			\hline 			
		\end{tabular}
	\end{center}
\end{table}
\begin{table}[H]
	\begin{center}
		\caption{The various estimates of $\theta_{\textnormal{y}}^S$ for $a=-1$.} 
		\vspace{.1cm} 
		\label{table-est2} 	\setlength{\tabcolsep}{8pt}
		\def\arraystretch{1.4}
		\begin{tabular}{|c|c|c|c|c|c|c|}
			\hline
			$\delta_{N,1}$& $\delta_{N,1}^{I3}$ & 	$\delta_{N,2}$& 	 	$\delta_{N,3}$ & 	$\delta_{N,3}^{I2}$ & 	$\delta_{N,4}$ & $\delta_{N,4}^{I3}$\\
			\cline{1-7} 
			131.4569	&401.8278 &607.9281&132.0856&401.8278&163.5922& 163.5922\\
			\hline 			
		\end{tabular}
	\end{center}
\end{table}

\section{Risk Comparisons of Estimators}
In this section, we compare the risk performance of the proposed estimators of $\theta_{\textnormal{y}}^S$, using the  loss function given in (\ref{loss1.1}).  For this purpose, a simulation study is performed using MATLAB software to compute the  values of risk of the various estimators.  20,000 simulation runs  with  different configurations of parameters are used to obtain the risk values.
 (For this purpose, a simulation study is performed using MATLAB software with 20,000 simulation runs  and   different configurations of parameters are used.
  Note that the estimator with the least average  risk values is preferable.  Further, the natural selection rule  $\boldsymbol{\psi}$ presented in Equation (\ref{sel-rule}) is used for achieving the aim of selecting the best bivariate normal population.  It is easy to see that, the risk of the proposed  estimators of $\theta_{\textnormal{y}}^S$  depend on the parameters  $\sigma_{xx}$, $\sigma_{yy}$, $\rho$, $a$ and $\theta^{(1)}=\left(  \theta_{x}^{(1)}, \theta_{y}^{(1)}\right)$, $\theta^{(2)}=\left(  \theta_{x}^{(2)}, \theta_{y}^{(2)}\right)$ (only through $\theta_x$ and $\theta_y$). So that, the risk functions are vary  for different combinations of these parameters. The computed values of risks of the various estimators of $\theta_{\textnormal{y}}^S$  are presented in Tables \ref{table:risk1}-\ref{table:risk6}, for different combinations of  $\theta^{(1)}$,  $\theta^{(2)}$, 
 and for $\sigma_{xx}= \sigma_{yy}= 2$,  $\rho  \in \left\{ -1,0,1 \right\}$, and  $a \in \left\{ -1,1 \right\}.$   Note that the computation of risk values was carried-out for other values of $a$ and $\rho$ but these values were omitted from the tables because  the same results were obtained.  The risk values of the hybrid estimator $\delta_{N,4}$ were calculated for $c=1$. In view of  the risk values in   Tables \ref{table:risk1}-\ref{table:risk6},  we present the following assessment of  the estimators of $\theta_{\textnormal{y}}^S$. 
\begin{itemize}
 	\item[(1)] For $a>0$ and $0<\rho \leq 1$, the imrpoved estimators  $\delta_{N,1}^{I1}$ and  $\delta_{N,2}^{I2}$ provide a considerable improvement upon the estimators $\delta_{N,1}$ and  $\delta_{N,2}$, respectively. The improved  estimators  $\delta_{N,3}^{I1}$ and  $\delta_{N,4}^{I1}$ have the same performance with the  estimators $\delta_{N,3}$ and  $\delta_{N,4}$, respectively, hence their risk values were omitted form Table \ref{table:risk1}. The improved estimator  $\delta_{N,2}^{I2}$ dominate all other estimators and 
 	 has the  least values of risk among other estimators. 

\item[(2)] 	For $a>0$ and $-1 \leq \rho < 0$, the improved estimators $\delta_{N,1}^{I3}$, $\delta_{N,2}^{I1}$, $\delta_{N,3}^{I3}$ and $\delta_{N,4}^{I2}$ perform better than their respective natural estimators. However, among all these estimators the improved estimator $\delta_{N,1}^{I3}$ has the best performance. 

\item[(3)] For $a>0$ and $\rho=0$, the improved estimator $\delta_{N,3}^{I4}$ provides a significant improvement upon the  estimator $\delta_{N,3}$. Also,  the  estimator $\delta_{N,3}^{I4}$ has better performance  than the  estimators $\delta_{N,2}$ and $\delta_{N,4}$ only   when $\theta_x \geq -0.2$ and $\theta_y \leq 0.2$. But,  when $\theta_x < -0.2$ and $\theta_y > 0.2$  the estimator $\delta_{N,2}$  performs better than $\delta_{N,3}^{I4}$. Further, the estimator $\delta_{N,2}$  dominates the three estimators $\delta_{N,1}$, $\delta_{N,3}$ and  $\delta_{N,4}$.  
\item[(4)] For $a<0$ and $0 < \rho \leq 1$, the estimator $\delta_{N,4}$ dominates the estimators $\delta_{N,2}$ and $\delta_{N,3}$, but, when $\theta_x$ and $\theta_y$ are very close to zero, $\delta_{N,3}$ dominates   $\delta_{N,4}$. The estimator $\delta_{N,1}$ dominates all the estimators of $\theta_{\textnormal{y}}^S$.  The improved  estimators  $\delta_{N,1}^{I3}$, $\delta_{N,3}^{I2}$ and  $\delta_{N,4}^{I3}$ have the same values of risk with the  estimators $\delta_{N,1}$ $\delta_{N,3}$ and  $\delta_{N,4}$, respectively, hence their risk values were omitted form Table \ref{table:risk4}.

\item[(5)] For $a<0$ and $-1 \leq \rho < 0$, the improved estimators  $\delta_{N,1}^{I2}$,  $\delta_{N,2}^{I2}$, $\delta_{N,3}^{I3}$  and $\delta_{N,4}^{I4}$ provide considerable improvement upon their respective natural estimators. However,  the improved estimator  $\delta_{N,2}^{I2}$ has the least risk values among all  these estimators. 

\item[(6)] For $a<0$ and $ \rho = 0$, the improved estimators $\delta_{N,1}^{I4}$, $\delta_{N,3}^{I4}$ and $\delta_{N,4}^{I5}$ provide only marginal improvement upon the  estimators $\delta_{N,1}$, $\delta_{N,3}$ and $\delta_{N,4}$, respectively.  The estimator  $\delta_{N,4}^{I5}$ domintes the other estimators when $\theta_x$ and $\theta_y$ are very close to zero, but when  $\theta_x$ and $\theta_y$ are not close to zero the  estimator $\delta_{N,2}$ dominates $\delta_{N,4}^{I5}$. 
 \end{itemize}
Based on the above observations, we conclude that, for $a>0$ and $ 0\leq \rho \leq 1$ the performance of the  estimator  $\delta_{N,2}^{I2}$ is satisfactory, hence is recommended for practical purposes.  For  $a>0$ and $ -1\leq \rho <0$, the estimator  $\delta_{N,1}^{I3}$ is recommended. For $a>0$ and $\rho=0$, the estimator $\delta_{N,3}^{I4}$ is recommended when $\theta_x \geq -0.2$ and $\theta_y \leq 0.2$ and  the estimator $\delta_{N,2}$ is recommended for other values of $\theta_{x}$ and  $\theta_{y}$. 
For $a<0$, the use of the natural estimator $\delta_{N,1}$ is recommended for $0< \rho \leq 1$ and the estimator $\delta_{N,2}^{I2}$ is recommended for $-1\leq \rho <0$. Also, for $a<0$ and $\rho =0$, the estimator $\delta_{N,4}^{I5}$ is recommended when $\theta_x$ and $\theta_y$ are very close to zero, and the estimator  $\delta_{N,2}$ is recommended when   $\theta_x$ and $\theta_y$ are not close to zero. 
  \section*{Acknowledgement}
The authors are thankful to Dr. A. A. Olosunde for providing a complete data set that used as an application in this paper.     

\pagebreak

\begin{table}[H]
	\begin{center}
		\caption{Organic and Inorganic copper-salt combinations observed data.} 
		\vspace{.1cm} 
		\label{Data} 	\setlength{\tabcolsep}{12pt}
		\def\arraystretch{1.3}
		\begin{tabular}{|cccc|cccc|}
			\hline
			\multicolumn{4}{|c|}{Organic Copper-Salt} & 	\multicolumn{4}{|c|}{Inorganic Copper-Salt}\\
			\hline
			\multicolumn{2}{|c}{Weight} & 	\multicolumn{2}{c|}{Cholestrol} &\multicolumn{2}{|c}{Weight} &\multicolumn{2}{c|}{Cholestrol}    \\
			\cline{1-8} 
			56.08 &  56.34	 & 60.73 & 66.03 	& 52.67 	& 53.17  & 164.23&  167.42 \\
			56.61 &  56.87& 71.33 &  76.63 &53.67 &54.17 &    170.60 &     173.78\\ 
			57.13&  57.39&    81.86 & 81.93 & 54.67 & 55.17 &    176.96 &    180.14 \\ 
			57.65 & 57.92 & 81.93 & 87.16 & 55.67 &     56.17 &      183.32 &  186.51 \\ 
			58.18 & 58.44 &  92.46 & 92.52 & 56.67   &  57.17  &  189.69 &     192.87 \\ 
			58.70 & 58.96  &  97.76 &  97.82 & 57.67 &  58.17 &   196.05 & 199.24 \\
			59.23 & 59.45  &  103.06 & 103.11 &  58.67  &  59.17 & 202.42 & 205.60 \\ 
			59.75  & 60.01  & 108.36 & 108.41  &  59.67  & 60.17 & 208.78 &  211.96 \\
			60.27 & 60.54 & 113.66  &  113.70 &  60.67 &  61.17 &  215.14 &  218.33 \\ 
			60.80 & 61.06  & 118.96 & 119.00 &  61.67 &  62.17  &  221.52 &  224.69 \\ 
			61.32 & 61.58 & 124.26 &  124.30 &  62.67 &  63.17 &  224.85&   227.88 \\ 
			61.85 & 62.34  &  129.56  & 129.60  & 63.43 & 65.15  & 228.03 &    231.06\\
			62.11 & 61.85 & 134.86  &  134.89 & 65.67 & 63.43  &    228.01 &  224.83 \\ 
			61.58 & 61.32  & 140.16 & 140.19 &    62.93 &  62.43 &  221.65 &  218.46 \\ 
			61.06  &  60.80  &  1745.46  &   145.48 &    61.93  &   61.43  &  215.28 & 212.10 \\  
			60.54 &  60.27  & 150.76  &150.78 &  60.93  &60.43&   208.92 &   205.74 \\
			60.01  &59.75   &156.06&   156.08 &59.93  & 59.43   & 202.56    & 199.37 \\
			59.49  & 59.23  &  161.36 &   161.37  & 58.93  &    58.43 & 196.19  &   193.01 \\
			59.00 &  58.70  & 166.66  &   166.67  & 57.93   & 57.43  &  189.83  &   186.65 \\
			58.44  &58.18    &171.96   & 171.97 & 56.93   &56.43  & 183.46   &180.28 \\
			57.92  & 57.65    & 177.26 &  177.26  &55.93   &55.43     &177.10  & 173.72 \\
			57.39  & 57.13 & 182.56  &  182.56 &54.93 & 54.43  &   170.74 &  167.55 \\ 
			56.87 & 56.61  & 182.56  & 187.86 & 53.93 &  53.43   & 164.37  &   161.19  \\ 
			56.34  & 56.08  & 187.86  &  193.16 & 52.93  &   52.43 & 158.01 &  154.83	 \\
			\hline 			
		\end{tabular}
	\end{center}
\end{table}
\begin{table}[H]
	\begin{center}
		\caption{Risk values of the various estimators of $\theta_{\textnormal{y}}^{S}$ for  $a=1, \, \sigma_{xx}=\sigma_{yy}=2, \, \rho = 1$}  \vspace{.5cm}
		\label{table:risk1} \setlength{\tabcolsep}{10pt}
		\def\arraystretch{1.5}
		\begin{tabular}{|cc|cccccc|}
			\hline
			\cline{1-8}$\theta^{(1)}$& $\theta^{(2)}$& 
			$\delta_{N,1}$& $\delta_{N,1}^{I1}$ & $\delta_{N,2}$ &$ \delta_{N,2}^{I2}$ &$ \delta_{N,3}$  &$ \delta_{N,4}$  \\
			\hline
			(0.2,2)	&		(2,0.2)	&
			2.6462	&	1.7312	&   0.9743	&	0.8315&	 3.7656&  	       2.5723
			\\
			(0.4,1.8)	&		(1.8,0.4)&	 2.6458 	&  1.6916	&	      1.0069
			&  0.8184	&	     3.4462 	&        	    2.7156
			\\
			(0.6,1.6)	&		(1.6,0.6)	&	  2.6915& 1.6833	&	     1.0001
			&	      0.7655	&	        2.9903 &     	     2.3291
			\\
			(0.8,1.4)	&		(1.4,0.8)	&	  2.9213	&	 1.7103	&	      1.0110	&	    0.7371	&    2.4982	&       2.2103  	\\
			(1,1.2)	&		(1.2,1)	&	 2.7271&	
			1.5697	&	      1.0090
			&    0.7050 	&	  2.5228	&      2.2934  	\\
			(0,0)	&		(0,0) &  2.6445 &	 2.6445&       1.0463						
			&        1.0463			
			&     2.4251						
			&         2.3460 	\\
			(1.2,1)	&		(1,1.2) &  2.8050&    1.5987	&    0.9524	&    0.7978& 2.6302	&    2.3102   \\			
			(1.4,0.8)    &		(0.8,1.4)	&	 3.0249 &1.7512&	     0.9608			&	      0.8304		&     2.5218&         2.3739		\\
			(1.6,0.6)	&		(0.6,1.6)	&	  2.7631 	&	 1.6928	&	0.9847		&      0.8421			&	     2.5866				 &             2.4471		 \\
			(1.8,0.4)	&(0.4,1.8)&2.5929	& 1.6722 & 0.9783&   0.8634	&   2.8642 &         2.5125	\\
			(2,0.2)	&		(0.2,2)	& 2.5158	& 1.7084 	&	    0.9673	&	    0.8592	&	  3.4061	&         2.6093	 	\\
			\hline
		\end{tabular}
	\end{center}
\end{table}

\begin{table}[H]
	\begin{center}
		\caption{Risk values of the various estimators of $\theta_{\textnormal{y}}^{S}$ for  $a=1, \, \sigma_{xx}=\sigma_{yy}=2, \, \rho = -1$}  \vspace{.5cm}
		\label{table:risk2} \setlength{\tabcolsep}{7pt}
		\def\arraystretch{1.5}
		\begin{tabular}{|cc|cccccccc|}
			\hline
			\cline{1-10}$\theta^{(1)}$& $\theta^{(2)}$& 
			$\delta_{N,1}$& $\delta_{N,1}^{I3}$ & $\delta_{N,2}$ &$ \delta_{N,2}^{I1}$ &$ \delta_{N,3}$ &$ \delta_{N,3}^{I3}$ &$ \delta_{N,4}$ & $ \delta_{N,4}^{I2}$ \\
			\hline
			(0.2,2)	&		(2,0.2)	&
			0.8444&	  0.8311	&  1.0076&  0.9003&   1.1567& 
			1.0309
			&	    1.1740 &  1.0267
			\\
			(0.4,1.8)	&		(1.8,0.4)&	  0.8144&  0.7603&1.0551&
			0.8648	&  1.0014& 0.8922&	    1.0819
			&  0.9385
			\\
			(0.6,1.6)	&	(1.6,0.6)	&0.7200&  0.6401	& 1.0794&0.7991&  0.7685& 0.7001 &	0.8720	&  0.7522
			\\
			(0.8,1.4)	&		(1.4,0.8)	&    0.6935	&  0.6061	&  1.1007
			&  0.8165&0.6789&  0.6317&  0.8052&   0.7123
			\\
			(1,1.2)	&		(1.2,1)	&	0.6668&	  0.5711	&1.1303&  0.7972	&0.6010	& 0.5749& 0.7359 &  0.6551			 
			\\
			(0,0)	&		(0,0) &  0.6636&  0.5244&1.1273&  0.6705		&0.5974 & 0.5728	&  0.7312
			&   0.6167
			\\
			(1.2,1)	&		(1,1.2) &   0.6759 &	0.5883&	 1.1219
			&        0.8072			
			&	0.6331	& 0.6077 &   0.7545
			&    0.6796
			\\	(1.4,0.8)    &		(0.8,1.4)	&  0.6741&	    0.5921
			&  1.0815 & 0.8039 &0.6631 &  0.6163
			& 0.7860  &  0.6968
			\\
			(1.6,0.6)	&		(0.6,1.6)	& 0.7262 &0.6411 &   1.0883 &0.7959
			&    0.7793 &   0.7183&      0.8866
			& 0.7703
			\\
			(1.8,0.4)	&		(0.4,1.8)	& 0.8091 &0.7601 &	 1.0564 &   0.8658	&0.9857 & 0.8826
			&1.0557
			& 0.9217			 
			\\
			(2,0.2)	&		(0.2,2)	&  0.8739& 0.8523& 
			1.0286 &  0.9121 &1.1834&  1.0477&1.2171
			& 1.0514
			\\
			\hline
		\end{tabular}
	\end{center}
\end{table}

\begin{table}[H]
	\begin{center}
		\caption{Risk values of the various estimators of $\theta_{\textnormal{y}}^{S}$ for  $a=1, \, \sigma_{xx}=\sigma_{yy}=2, \, \rho = 0$}  \vspace{.5cm}
		\label{table:risk3} \setlength{\tabcolsep}{10pt}
		\def\arraystretch{1.5}
		\begin{tabular}{|cc|ccccc|}
			\hline
			\cline{1-7}$\theta^{(1)}$& $\theta^{(2)}$& 
			$\delta_{N,1}$& $\delta_{N,2}$ & $\delta_{N,3}$ &$ \delta_{N,3}^{I4}$ &$ \delta_{N,4}$  \\
			\hline
			(0.2,2)	&		(2,0.2)	&
			1.8102
			&      1.0294
			& 2.8037 &    0.2382  &       1.7059		     
			\\
			(0.4,1.8)	&	(1.8,0.4)&	    1.7510	&  1.0135&   2.2707	& 0.2913&    1.4270
			\\
			(0.6,1.6)	&		(1.6,0.6)	&    1.7064&0.9899&	   1.7890
			&	0.3291	&	   1.2453\\
			(0.8,1.4)	&		(1.4,0.8)	&   1.7238	&	 1.0044
			& 1.5421&   0.5055& 1.0612	\\
			(1,1.2)	&(1.2,1)	&	 1.6948&  0.9882&	 1.4484 & 0.6394
			&  1.0201
			\\
			(0,0)	&		(0,0) &  1.7815	&0.9668& 1.4718		
			& 0.7055&  
			1.0048\\
			(1.2,1)	&		(1,1.2) &     1.7621
			&   0.9876&  1.4629& 0.8513&	1.0175	\\			
			(1.4,0.8)    &		(0.8,1.4)	& 1.7868
			&1.0149			
			&   1.6371	&1.1382&   1.1297 \\
			(1.6,0.6)	&		(0.6,1.6)	&      1.8065
			&	 1.0233	&	1.8261	& 1.4093&	 1.2745 \\
			(1.8,0.4)	&		(0.4,1.8)	&   1.7162
			&1.0005&2.1955	& 1.8950&   1.3780	\\
			(2,0.2)	&		(0.2,2)	&  1.7324	& 
			1.00108	&    2.7087&   2.5046&	1.6951 \\
			\hline
		\end{tabular}
	\end{center}
\end{table}

\begin{table}[H]
	\begin{center}
		\caption{Risk values of the various estimators of $\theta_{\textnormal{y}}^{S}$ for  $a=-1, \, \sigma_{xx}=\sigma_{yy}=4, \, \rho = 1$}  \vspace{.5cm}
		\label{table:risk4} \setlength{\tabcolsep}{10pt}
		\def\arraystretch{1.5}
		\begin{tabular}{|cc|cccc|}
			\hline
			\cline{1-6}$\theta^{(1)}$& $\theta^{(2)}$& 
			$\delta_{N,1}$ & $\delta_{N,2}$ &$ \delta_{N,3}$  &$ \delta_{N,4}$  \\
			\hline
			(0.2,2)	&		(2,0.2)	&
			0.8470&1.0127	&  5.7682&0.8711\\
			(0.4,1.8)	&		(1.8,0.4)&  0.7248&1.0865 & 2.5465	&   0.8804\\
			(0.6,1.6)	&		(1.6,0.6)	&0.7089
			& 1.0710&1.3354&  0.8717\\
			(0.8,1.4)	&		(1.4,0.8)	&   0.7276
			&1.0884&  0.8350&      0.7409 \\
			(1,1.2)	&		(1.2,1)	&  0.6622
			&1.0792& 0.6940&	   0.7321\\
			(0,0)	&		(0,0) &     0.7230
			&     1.0702	&  0.8046&  0.8953 \\
			(1.2,1)	&		(1,1.2) &      0.6552
			& 1.0754&   0.6857
			&    
			0.7314  \\			
			(1.4,0.8)    &		(0.8,1.4)	&    0.7326	&	     1.0777	&	0.8444	&    0.7707\\
			(1.6,0.6)	&		(0.6,1.6)	&0.7397& 1.0828&   1.2542&    0.7920
			\\
			(1.8,0.4)	&		(0.4,1.8)	&0.7428	& 1.0912& 2.5213&   0.9157\\
			(2,0.2)	&		(0.2,2)	&   0.7429&  1.0854	&   5.7609& 0.9237\\
			\hline
		\end{tabular}
	\end{center}
\end{table}

\begin{table}[H]
	\begin{center}
		\caption{Risk values of the various estimators of $\theta_{\textnormal{y}}^{S}$ for  $a=-1, \, \sigma_{xx}=\sigma_{yy}=4, \, \rho = -1$}  \vspace{.5cm}
		\label{table:risk5} \setlength{\tabcolsep}{6pt}
		\def\arraystretch{1.5}
		\begin{tabular}{|cc|cccccccc|}
			\hline
			\cline{1-10}$\theta^{(1)}$& $\theta^{(2)}$& 
			$\delta_{N,1}$& $\delta_{N,1}^{I2}$ & $\delta_{N,2}$ &$ \delta_{N,2}^{I2}$ &$ \delta_{N,3}$ &$ \delta_{N,3}^{I3}$ &$ \delta_{N,4}$ & $ \delta_{N,4}^{I4}$ \\
			\hline
			(0.2,2)	&		(2,0.2)	&
			2.5510	&	  1.5050&	0.9719& 0.8279&5.4113&  2.6828  &	      2.2766&   1.4729 \\
			(0.4,1.8)	&		(1.8,0.4)& 2.7091& 1.5983& 0.9564
			& 0.7988& 5.8670&  2.7114&	 2.3859&  1.4868\\
			(0.6,1.6)	&		(1.6,0.6)	&	 2.6500
			&    1.3048	&	 0.8987	&  0.6791&5.9427& 2.1990& 2.3006
			& 1.1539\\
			(0.8,1.4)	&		(1.4,0.8)	&	 2.9947&     1.5611& 0.9747	&  0.6968	& 6.7696&  3.0248&  2.5833&  1.4237\\
			(1,1.2)	&		(1.2,1)	&	 2.6574 &	1.3893 &	0.8401& 0.6219	&6.1065	&  2.6089 &   2.2276 &1.2886
			\\
			(0,0)	&		(0,0) &   2.5950&	 0.9487&0.8343
			&  0.5545
			&    6.0156&  1.3589	& 2.1910& 0.8812 \\
			(1.2,1)	&		(1,1.2) &  2.7163&	  1.3609 &	0.8578	&  0.6191	&	 6.2608	&    2.5868 &    2.2753  &    1.2710  \\			
			(1.4,0.8)    &		(0.8,1.4)	&  2.7172
			&1.4031
			& 0.8875&0.6452&  6.2037 &   2.7325&2.3345	& 1.2680	\\
			(1.6,0.6)	&		(0.6,1.6)	&  2.6410
			&	1.2674	& 0.8919	&     0.6566	&5.8921 &  2.1696	&  2.2734&1.1366 \\
			(1.8,0.4)	&		(0.4,1.8)	& 2.6376 & 1.5342 &	  0.9428&  0.7901	&   5.7970&   2.6678 & 2.3202 &     1.4445 	\\
			(2,0.2)	&		(0.2,2)	&2.5747	&   1.4634 &  0.9800
			&0.8183&  5.5067	&	 2.5940 & 2.2982	&   1.4377	\\
			\hline
		\end{tabular}
	\end{center}
\end{table}
\begin{table}[H]
	\begin{center}
		\caption{Risk values of the various estimators of $\theta_{\textnormal{y}}^{S}$ for  $a=-1, \, \sigma_{xx}=\sigma_{yy}=4, \, \rho = 0$}  \vspace{.5cm}
		\label{table:risk6} \setlength{\tabcolsep}{7pt}
		\def\arraystretch{1.5}
		\begin{tabular}{|cc|cccccc|}
			\hline
			\cline{1-8}$\theta^{(1)}$& $\theta^{(2)}$& 
			$\delta_{N,1}$& $\delta_{N,1}^{I4}$ & $\delta_{N,2}$ &$ \delta_{N,3}$  &$ \delta_{N,4}$ &$ \delta_{N,4}^{I5}$ \\
			\hline
			(0.2,2)	&		(2,0.2)	&
			1.6912	&   1.6477	&  0.9909 & 80.5726&  1.3579&	 1.2440 \\
			(0.4,1.8)	&		(1.8,0.4)&	   1.7057 &   
			1.6595 &0.9975& 63.5385	&    1.2313	& 1.1164 \\
			(0.6,1.6)	&		(1.6,0.6)	&	  1.7053	&	    1.6616
			&	0.9974 &	      92.9430
			&    1.1331 &    1.0304  \\
			(0.8,1.4)	&		(1.4,0.8)	&     1.7269
			&  1.6796	& 1.0031 &  39.0709		&   1.0501
			&     0.9585\\
			(1,1.2)	&		(1.2,1)	&	    1.6670 
			&     1.6183	&	 0.9756	&    24.7978		&     1.0183&    0.9209	\\
			(0,0)	&		(0,0) &      1.6932
			&	     1.6436	&     0.9835 &   24.7124 &     0.9684	&    0.8822	\\
			(1.2,1)	&		(1,1.2) &     1.7682 
			&  1.7364 &	1.0191 	&        47.2484 	&     0.9901 &   0.8853  \\			
			(1.4,0.8)    &		(0.8,1.4)	&   1.7416			 
			& 1.7215	&   0.9995	& 34.1616 	&    1.0370&  0.9302 \\
			(1.6,0.6)	&		(0.6,1.6)	&  1.6809
			&1.6567	&  1.0217 &    53.2835
			& 1.1860
			& 1.0729 	\\
			(1.8,0.4)	&		(0.4,1.8)	&	    1.6833
			&	1.6786 & 0.9869 	&   136.5534	  &  1.2239 & 1.0576	\\
			(2,0.2)	&		(0.2,2)	&   1.6913	&   1.6913 &   1.0117	&   76.2712&   1.4199	&   1.2616\\
			\hline
		\end{tabular}
	\end{center}
\end{table}
\end{document}